\pdfoutput=1
\documentclass[10pt]{extarticle}

\usepackage[left=0.8in,right=0.8in,top=0.8in,bottom=1.2in]{geometry}
\usepackage{amsmath, amssymb, amsthm, latexsym, amsfonts, indentfirst, xcolor, mathtools, manfnt, microtype, subfig, cite, graphicx, extarrows}
\usepackage[utf8]{inputenc} 
\usepackage[english]{babel}

\usepackage[breaklinks]{hyperref}
\usepackage{cleveref}
\hypersetup{
	colorlinks = true, 
	urlcolor = cyan, 
	linkcolor = teal, 
	citecolor = cyan 
}

\let\OLDthebibliography\thebibliography
\renewcommand\thebibliography[1]{
	\OLDthebibliography{#1}
	\setlength{\parskip}{0pt}
	\setlength{\itemsep}{0pt plus 0.3ex}
}

\newtheorem{Theorem}{Theorem}
\newtheorem{Claim}{Claim}

\newtheorem{Lemma}{Lemma}

\newtheorem{Proposition}{Proposition}

\newcommand{\R}{\mathbb R}

\newcommand{\N}{\mathbb N}
\newcommand{\Q}{\mathbb Q}
\newcommand{\x}{\mathbf x}
\newcommand{\y}{\mathbf y}

\newcommand{\MR}{\smash{ \xrightarrow[\smash{\raisebox{0.7ex}{\tiny $\mathbf{MR}$}}]{}}}
\newcommand{\MRR}{\smash{ \xrightarrow[\smash{\raisebox{0.7ex}{\tiny $\mathbf{MR}$}}]{r}}}

\newcommand{\nMRR}{\smash{ \not\xrightarrow[\smash{\raisebox{0.7ex}{\tiny $\mathbf{MR}$}}]{r}}}

\begin{document}

\title{Canonical theorems in geometric Ramsey theory}
\author{
Panna Geh\'er\thanks{E\"otv\"os Lor\'and University, Budapest, Hungary. Email:~\href{mailto:geher.panna@ttk.elte.hu}{\tt geher.panna@ttk.elte.hu}.}
	\and
	Arsenii Sagdeev\thanks{KIT, Karlsruhe, Germany and Alfr\'ed R\'enyi Institute of Mathematics, Budapest, Hungary. Email:~\href{mailto:sagdeevarsenii@gmail.com}{\tt sagdeevarsenii@gmail.com}.}
	\and
	G\'eza T\'oth\thanks{Alfr\'ed R\'enyi Institute of Mathematics, Budapest, Hungary. Email:~\href{mailto:geza@renyi.hu}{\tt geza@renyi.hu}.}
}

\maketitle

\begin{abstract}
	In Euclidean Ramsey Theory usually we are looking for monochromatic  configurations in the Euclidean space, whose points are colored with a fixed number of colors. In the canonical version, the number of colors is arbitrary, and we are looking for an `unavoidable' set of colorings of a finite configuration, that is a set of colorings with the property that one of them always appears in any coloring of the space. This set definitely includes the monochromatic and the rainbow colorings. In the present paper, we prove the following two results of this type. First, for any acute triangle $T$, and any coloring of  $\mathbb{R}^3$, there is either a monochromatic or a rainbow copy of $T$. Second, for every $m$, there exists a sufficiently large $n$ such that in any coloring of  $\mathbb{R}^n$, there exists either a monochromatic or a rainbow $m$-dimensional unit hypercube. In the maximum norm, $\ell_{\infty}$, we have a much stronger statement. For every finite $M$, there exits an $n$ such that in any coloring of $\R_\infty^n$, there is either a monochromatic or a rainbow isometric copy of $M$.
\end{abstract}

\textbf{Key words:} Euclidean Ramsey theory, canonical Ramsey theorem, colorings of the space

\textbf{Mathematics Subject Classification:} 05D10, 05C55

\section{Introduction}

A typical problem in \textit{Ramsey theory}~\cite{GRS91} is to show that every $r$-coloring 
of a sufficiently large `domain' contains a monochromatic copy of some fixed `configuration'. 
A classical result in the field is the van der Waerden theorem~\cite{Wae27}, 
which states that for all $k, r \in \N$, there exists $n$ such that every $r$-coloring of $[n]$ 
contains a monochromatic $k$-term arithmetic progression. Each problem in Ramsey theory 
has its \textit{canonical} counterpart, in which the goal is to list a minimal set $S$ of colorings which is  `unavoidable' regardless of the number of colors in use. That is, in any coloring of the domain we have a configuration whose coloring belongs to $S$. A coloring of a configuration is called {\em rainbow} if all of its point are of different color.
For example, the canonical van der Waerden~\cite{EG80} theorem states that for all $k \in \N$,  there exists $n$ such that every coloring of $[n]$ contains either a monochromatic or a \textit{rainbow} $k$-term arithmetic progression. For modern exposition and further generalizations of this 
result, see~\cite{GKP12,Gir20} and the references therein. We note that the minimal sets of unavoidable colorings often contain  \textit{other} colorings along with the monochromatic and the rainbow. In fact, this is the case for canonical versions of both the Ramsey theorem due to Erd\H os and Rado~\cite{ER50}, see also~\cite{JM00,AL05,W06,JJL03}, and the Hales--Jewett theorem~\cite{PV83,NPRV85}, see \Cref{Sec_HJ} for the exact statement.

In their celebrated trilogy~\cite{EGMRSS1, EGMRSS2, EGMRSS3}, Erd\H{o}s, Graham, Montgomery, Rothschild, Spencer and Straus initiated a systematic study of the following general problem in geometric Ramsey theory. For $n,r \in \N$ and $A \subset \R^n$, does every $r$-coloring of $\R^n$ contain a monochromatic isometric copy of $A$? They also introduced the notation 
$\R^n \xrightarrow{r} A$ for an affirmative answer on the latter question. 
Note that only a few results in the field are tight. Even the innocent-looking famous special case, when $n=2$ and 
$A=I$ is a fixed two-point set is open. It is equivalent to the notorious Hadwiger-Nelson problem which asks for the chromatic number of the plane. After 75 years of extensive study, there are still two values of $r$, namely $r=5$ and $r=6$, for which we do not know whether 
$\R^2 \xrightarrow{r} I$ or not, see~\cite{degrey,EI20} and the survey in~\cite{Soifer}. 

Another major open problem in the field is to describe all \textit{Ramsey sets}, 
that is sets $A$ with the property that for all $r \in \N$ there is a sufficiently large $n$ with $\R^n \xrightarrow{r} A$. It is known that every Ramsey set is spherical \cite[Theorem~13]{EGMRSS1} and finite \cite[Theorem~19]{EGMRSS2}. We refer the reader to~\cite{LRW12} for the discussion about if these two necessary conditions might also be sufficient. From the other direction, it is known \cite{Kriz1991,Cant2007} that (the vertex sets of) all regular polytopes are Ramsey. The vertex set $A$ of a box or a simplex in arbitrary dimension has an even stronger \textit{exponentially Ramsey} property: there exists a positive constant $c(A)$ such that for all $r \in \N$ and $n\sim c(A)\log r$, we have $\R^n \xrightarrow{r} A$, see~\cite{FR90}. For the currently best bounds on $c(A)$, we refer the reader to \cite{KSZ23} and the references therein. 

To the best of our knowledge, the study of canonical results in Euclidean Ramsey theory was initiated only recently. 
In~\cite{MOW22} the authors proved that every $r$-coloring of $\R^n$ contains either a monochromatic or a rainbow copy of a fixed rectangle for $n=\Theta(r)$. However, this result is far from  optimal, because $n=\Theta(\log r)$ dimensions are already sufficient to ensure a monochromatic copy of a given rectangle as we mentioned above. In fact, it seems hard to avoid rainbow copies of a given set $A$ whenever the number of colors in use is much larger than the cardinality $|A|$. Thus it feels only natural to expect that for most of the Ramsey sets $A$, if not for all of them, there exists a finite $n$ such that every coloring of $\R^n$ contains either a monochromatic or a rainbow copy of $A$ regardless of the number of colors in use. In the spirit of classic arrow notation, we denote the latter claim by $\R^n \MR A$ for short. Recently Cheng and Xu~\cite{CX23} found the first evidence supporting this intuition. Using our notation, their theorem can be stated as follows:

\begin{Theorem}[Cheng, Xu~\cite{CX23}] \label{th_CX1}
	\phantom{a} \vspace{-1mm}
	\begin{itemize}
		\item For every acute triangle $T$, there exists $n$ such that $\R^n \MR T$.
	\end{itemize}
	\phantom{a} \vspace{-8mm}
	\begin{itemize}
		\item For every right triangle $T$, we have $\R^3 \MR T$.
	\end{itemize}
\end{Theorem}

In fact, they proved a stronger statement, which also covers simplices of larger dimension and obtuse triangles that are not `too flat'. In the present paper, we refine their technique and strengthen the above result.

\begin{Theorem} \label{theorem_triangle}
    For every acute triangle $T$, we have $\R^3 \MR T$.
\end{Theorem}

Another canonical result obtained by Cheng and Xu in~\cite{CX23} covers the case of a square $I^2$.

\begin{Theorem}[Cheng, Xu~\cite{CX23}] \label{I2}
There exists $n$ such that $\R^{n} \MR I^2$.
\end{Theorem}

Their method probably can be generalized to a cube $I^3$ but perhaps not further, because for $m\ge 4$, the diameter of the (vertex set of an) $m$-dimensional hypercube $I^m$ becomes too large in comparison with its sidelength, which seems to be an obstacle. Here we present a completely different approach to bypass this limitation. 

\begin{Theorem} \label{cube}
For every $m \in \N$, there exists $n$ such that $\R^{n} \MR I^m$.
\end{Theorem}

Euclidean Ramsey theory questions (along with the corresponding notation) can be translated into non-Euclidean norms in a straightforward manner, with the only difference being the definition of an isometric copy\footnote{For a norm $N$ on $\R^n$, a set $A' \subset \R^n$ is an \textit{$N$-isometric copy} of $A \subset \R^n$ if there exists a bijection $f: A \to A'$ such that $\|\x-\y\|_N = \|f(\x)-f(\y)\|_N$ for all $\x,\y \in A$.}, see~\cite{EFI21,Dav23,FGST23,Geher23,KirS23,Kup11,Rai04,Vor23}. The case of an $n$-dimensional space $\R_\infty^n$ equipped with the max-norm\footnote{Recall that for $\x = (x_1,\dots,x_n) \in \R^n$, its \textit{max-norm} is defined by $\|\x\|_\infty = \max_{i}|x_i|$} seems to stand out here due to the following simple classification of finite (exponentially) Ramsey sets. It is not hard to deduce from the Hales--Jewett theorem that for every $r \in \N$ and a finite $M \subset \R^d$, there exists $n$ such that $\R_\infty^n \xrightarrow{r} M$. In fact, the latter claim holds already for $n=\Theta( \log r)$ as it was shown in~\cite{KupS21,FKS}. In the present paper, we deduce the following general canonical result in max-norm Ramsey theory from the canonical Hales--Jewett theorem.

\begin{Theorem} \label{maxnorm1}
For every finite $M \subset \R^d$, there exists $n$ such that $\R_\infty^n \MR M$.
\end{Theorem}

\noindent \textbf{Paper outline.}
In \Cref{Sec2}, we prove Theorems~\ref{theorem_triangle} and~\ref{cube}. In \Cref{Sec3}, we formally state the canonical Hales--Jewett theorem and use it to prove \Cref{maxnorm1}. Finally, in \Cref{Sec4}, we make some further comments
and state more open problems.

\section{Euclidean norm} \label{Sec2}

\subsection{Triangles -- proof of Theorem~\ref{theorem_triangle}}

Let $T$ be an acute triangle with sides $a$, $b$, $c$. 
Consider an arbitrary coloring of $\R^3$ and assume for a  contradiction that there is no  monochromatic or rainbow copy of $T$. Then some two points, which we denote by $Y_1$ and $Y_2$, are of different colors. Take a sequence $Y_1=X_0, X_1, \ldots, X_k=Y_2$ such that consecutive points are at distance $c$.  Then, there are two consecutive points, $X_i=A$ and $X_{i+1}=B$, of different colors, say, $A$ is red and $B$ is blue. 

We say that a pair of points $(X, Y)$ is a {\em good pair} if they are at distance $c$, 
one of them is red while the other one is blue. In particular, $(A, B)$ is a good pair. 
Let $X$, $Y$, $U$, $V$ be points such that $|XY|=|UV|=c$, $|XU|=|YV|=b$, $|XV|=|YU|=a$. That is, $XYUV$ is a simplex, all of whose faces are congruent to $T$. Then we say that the pair $(U,V)$ is an {\em extension} of the pair $(X, Y)$.

\begin{Claim}\label{goodpair}
	If $(X, Y)$ is a good pair, and $(U, V)$ is its extension, then $(U, V)$ is also a good pair.
\end{Claim}

\begin{proof}
	Suppose without loss of generality that $X$ is red and $Y$ is blue. Since both $UXY$ and $VXY$ are congruent to $T$, both $U$ and $V$ are red or blue. But $UVX$ and $UVY$ are also congruent to $T$, so one of $U, V$ is red, and the other one is blue.
\end{proof}

\begin{Claim}
   Every good pair has an extension.
\end{Claim}
 
\begin{proof}
	
Let $(A,B)$ be a good pair. Take a ball of radius $b$ (resp.\ $a$) about $A$ (resp.\ $B$). Their intersection is the circle $C_1$, its plane is perpendicular to $AB$, and its center is on the segment $AB$ since $T$ is an acute triangle. Similarly, take a ball of radius $b$ (resp.\ $a$) about $B$ (resp.\ $A$). Their intersection is the circle $C_2$, its plane is perpendicular to $AB$, and its center is on the segment $AB$. Fix an arbitrary $C \in C_1$, and let $X$ (resp.\ $Y$) be the point of $C_2$ closest (resp.\ farthest)
from $C$. 
We have $|CX|<|AB|$. Indeed, if $a=b$, then it is trivial because $C=X$. Otherwise, it follows from the fact that $CX$ and $AB$ are parallel and the orthogonal projections of $X$ and $C$ to the line $AB$ are inside the segment $AB$, see \Cref{fig_extension}. The point $Y$ is the reflection of $C$ through the midpoint $O$ of the segment $AB$. 
Since $\measuredangle ACB=\measuredangle AYB< 90^{\circ}$, both
$C$ and $Y$ are outside the ball of center $O$ and radius $|AB|/2$. Consequently $|CY|>|AB|$. 
Therefore, there is a point $D$ on $C_2$ with $|CD|=|AB|$. The pair $(C,D)$ is the desired extension.
\end{proof}
    
\begin{figure}[htp]
\centering
\includegraphics[width=12cm]{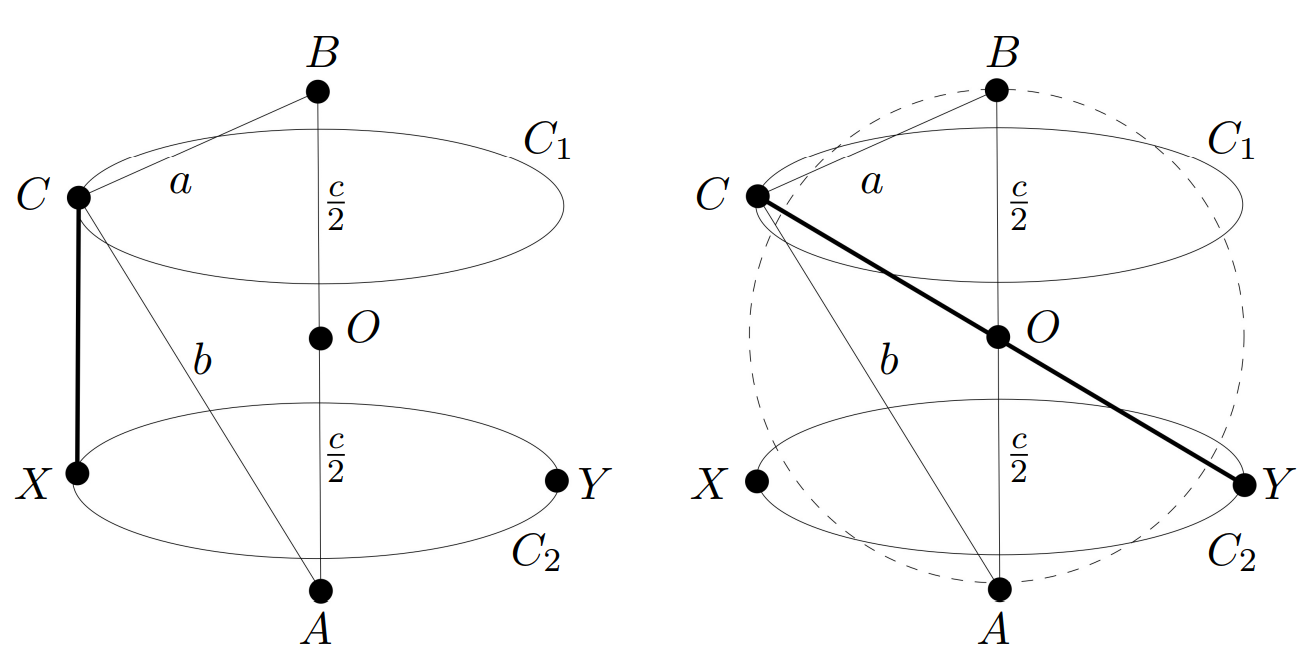}
\caption{Segment $AB$ is longer than $CX$ and shorter than $CY$.}
\label{fig_extension}
\end{figure}

\begin{Lemma}\label{rotations}
Let $l_1$ and $l_2$ be two skew lines in $\R^3$.
Suppose that the set  $S\subset \R^3$ is invariant under rotations about $l_1$ and $l_2$. 
Then $S=\varnothing$ or $S=\R^3$.
\end{Lemma}

\begin{proof}
We will use the following notation:
$\mbox{E}^+$ is the group of all orientation-preserving isometries of $\R^3$.
For any point $P$, $\mbox{SO}(P)$ is the group of all orientation-preserving isometries that do not move the point $P$.
For any line $l$, $\mbox{R}(l)$ is the group of all rotations about $l$.
The subgroup of translations of $\R^3$ is denoted by $\mbox{T}$. 

Suppose that $l_1$ and $l_2$ are two skew lines in $\R^3$ and $S$ is invariant for the groups $\mbox{R}(l_1)$ and $\mbox{R}(l_2)$. Then it is also invariant for the generated group $\mbox{G}=\langle \mbox{R}(l_1), \mbox{R}(l_2)\rangle$.
We claim that $\mbox{G}=\mbox{E}^+$, which is sufficient to show that either $S=\varnothing$ or $S=\R^3$. Indeed, if for some points $X,Y\in \R^3$, we have $X\in S$ but $Y\not\in S$, then $S$ is not invariant for a translation by the vector $\overrightarrow{XY}$, a contradiction.

Rotate the line $l_1$ about $l_2$. We get a hyperboloid $H_1$ which contains $l_1$, and its axis is $l_2$
(if $l_1$ and $l_2$ are orthogonal, our hyperboloid is degenerate, it is a plane minus a disc, but the argument works also in this case). Similarly, rotate $l_2$ about $l_1$. We get a hyperboloid $H_2$ which contains $l_2$, and its axis is $l_1$. It follows, that the two hyperboloids intersect. Therefore, there is a rotated copy $l_1'$ of $l_1$, and a rotated copy $l_2'$ of $l_2$, that intersect in a point $P$. It is not hard to see that a rotation about $l_1'$ is equal to a conjugate of the rotation about $l_1$ by the same angle with an element of $\mbox{R}(l_2)$ that transforms $l_1$ into $l_1'$. Hence $\mbox{R}(l_1')\subset \mbox{G}$. Similarly, $\mbox{R}(l_2')$ is a subgroup of $\mbox{G}$, and thus so is $\mbox{G}'=\langle \mbox{R}(l_1'), \mbox{R}(l_2')\rangle$.
  
We claim that  $\mbox{G}'=\mbox{SO}(P)$. The inclusion  
$\mbox{G}' \subseteq \mbox{SO}(P)$ is obvious. Take a unit sphere ${\cal S}$ of center $P$.
The groups  $\mbox{R}(l_1')$, $\mbox{R}(l_2')$, $\mbox{SO}(P)$ are orientation-preserving isometries of ${\cal S}$.
It is known \cite[Theorem~3.20]{Ced89} that each element of $\mbox{SO}(P)$ is a rotation about a line through $P$.
Let $Q_1$ (resp.\ $Q_2$)  be an intersection of the line $l_1'$ (resp.\ $l_2'$) and ${\cal S}$. Consider the orbit $O(Q_1)$ of the point $Q_1$ under the action of $\mbox{G}'$. It contains at least one point $Q_1'$ different from $Q_1$. As earlier, each element of $\mbox{R}(PQ_1')$ is a conjugate of the corresponding element of $\mbox{R}(l_1')$ with an element of $\mbox{G}'$ that transforms $Q_1$ into $Q_1'$.  Hence $\mbox{R}(PQ_1')\subset \mbox{G}'$. If $\measuredangle Q_1PQ_1'<90^{\circ}$, then rotate $Q_1$ about the line 
$PQ_1'$ by $180^{\circ}$, let $Q_1''$ be the image. Then $Q_1''$ is also in $O(Q_1)$, and $\measuredangle Q_1PQ_1''=2\measuredangle Q_1PQ_1'$. By repeating this procedure if necessary, we get a point 
$V_1\in O(Q_1)$ such that $\measuredangle Q_1PV_1\ge 90^{\circ}$. All elements $\mbox{G}'$ are finite products of rotations about lines $l_1'=Q_1P$ and $l_2'=Q_2P$.  So, we can decrease all rotation angles continuously to $0$. This shows that $\mbox{G}'$ is connected, and we can find $V_1'\in O(Q_1)$ such that $\measuredangle Q_1PV_1'=90^{\circ}$. But the rotations about two orthogonal axes, $\mbox{R}(PQ_1)$ and $\mbox{R}(PV_1')$, act transitively on ${\cal S}$. Therefore, we can get all rotation axes, and, consequently, all rotations. This proves that  $\mbox{G}'=\mbox{SO}(P)$.

Consider the lines $l_1$ and $l_1'$. Line $l_1$ does not contain $P$, but $l_1'$ does. Applying an appropriate element of $\mbox{SO}(P)$, we can transform $l_1'$ into a line $l_1''$ which is parallel to $l_1$, but they are not identical.
Clearly, $\mbox{R}(l_1'')\subset \mbox{G}$. Take the composition of rotation about $l_1$ by angle $\alpha$ 
and a rotation about $l_1''$  by angle $-\alpha$. The result is a translation $\mbox{T}(\alpha)$. 
As $\alpha$ goes to $0$, the translation distance of $\mbox{T}(\alpha)$ also goes to $0$ continuously. 
So, we can obtain a translation by any given distance, but the direction changes. Taking its conjugate with the elements of $\mbox{SO}(P)$, we can change the directions of the translations arbitrarily, therefore, we can obtain 
any translation, so $\mbox{T}\subset \mbox{G}$. Finally, 
$\mbox{E}^+=\mbox{T}\cdot \mbox{SO}(P)$, therefore,  $\mbox{G}=\mbox{E}^+$, as desired. 
\end{proof}

Return to the proof of 
Theorem~\ref{theorem_triangle}.
Let $S'\subset \R^3\times \R^3$ be the following set of pairs of points.
Let the pair $(A, B)$ be in $S'$, and take its closure under the extension operation.
In other words, the elements of $S'$ are exactly those pairs that can be obtained by finitely many applications of the extension operation, starting with the pair $(A, B)$. 
By Claim  \ref{goodpair}, all pairs in $S'$ are good pairs. 
Let $S\subset \R^3$ be the set of all points included in at least one pair from $S'$. Suppose that $(C, D)$ is an extension of $(A, B)$, and let $(C', D')$ be a rotation of $(C, D)$ about the line $AB$, by an arbitrary angle. Then $(C',D')$ is also an extension of $(A, B)$. Therefore, the set $S$ is invariant under rotations about the lines $AB$ and $CD$. So we can apply Lemma \ref{rotations} to conclude that $S=\R^3$ or $S=\varnothing$. However, $S$ is certainly not empty, as $A, B \in S$, therefore, $S=\R^3$. Since all points of $S$ are either red or blue, we used only two colors for the whole space. Now the following result of Erd\H os et al.\ guarantees that some copy of $T$ is monochromatic, which completes the proof of \Cref{theorem_triangle} by contradiction.

\begin{Theorem}[Theorem 8 in \cite {EGMRSS1}]
For every triangle $T$, we have $\R^3 \xrightarrow{2} T$.
\end{Theorem}

\subsection{Hypercubes -- proof of Theorem~\ref{cube}}

For any $A_1 \subset\R^{d_1}, \ldots ,A_n \subset\R^{d_n}$, 
their {\em Cartesian product} 
$A_1\times \cdots\times A_n$ is $\left\{ (\x_1, \ldots, \x_n)\ :\ \x_i\in A_i\right\}\subset\R^{d_1+\cdots +d_n}$. 
If $A_1=\cdots =A_n=A$, then 
$A_1\times\cdots\times A_n$ is the {\em Cartesian power} of $A$ and it is denoted by $A^n$ for short. 
Let $I^m(a)$ be an $m$-dimensional hypercube of sidelength $\sqrt{a}$ and $S_m(a)$ be an $m$-vertex, that is, an ($m-1$)-dimensional 
regular simplex of sidelength $\sqrt{a}$. 
Considering the neighbors of a fixed vertex, it is easy to see that $I^{m}(1)$ contains $S_m(2)$ as a subset. We also need the following asymmetric generalization of the arrow notation: by $C \MR (A,B)$, we denote that claim that every coloring of $C$ contains either a monochromatic copy of $A$, or a rainbow copy of $B$.

The main result of this section is the following statement.

\begin{Proposition} \label{Prop1}
	For all $k,m \in \N$, there exists $n=n_1(k,m)$ such that $I^n(1) \MR (I^k(2^k), I^m(2^k))$.
\end{Proposition}

\Cref{cube} follows from the special case $k=m$, i.e., from the claim that $I^n(1) \MR I^m(2^m)$.
We prove \Cref{Prop1} by induction on $k$ for any fixed $m$. First we prove the base case. 

\begin{Proposition} \label{Prop2}
	For all $m \in \N$, there exists $n=n_2(m)$ such that $I^n(1) \MR (I(2), I^m(2))$.
\end{Proposition}
\begin{proof}
	We use the probabilistic method to prove the following technical statement.
	
	\vspace{2mm}
	\noindent
	\textit{For all $m,d\in \N$ such that $d>4^m$, the Cartesian power $S_d^{m}(2)$ satisfies $S_d^{m}(2) \MR (I(2), I^m(2))$.}
	\vspace{1mm}

This implies Proposition \ref{Prop2} since  $S_d^m(2)\subset I^{md}(1)$.
Consider an arbitrary coloring of $S_d^{m}(2)$ 
and suppose that there is no monochromatic copy of  $I(2)$.

Choose an ordered pair of points from each of the $m$ simplices uniformly and independently at random. 
The Cartesian product $C$ of these ordered 2-point sets is a copy of $I^m(2)$ with a natural bijection between their points. 
If $C$ is not rainbow, then some two of its vertices have the same color. 
For any pair of vertices of $I^m(2)$, the event that the corresponding points of $C$ are of the same color, is called a 
{\em bad event}. The \textit{type} of the bad event is the pair of vertices. 
Note that a non-rainbow $C$ can belong to several types at the same time. 

There are 
$2^{m}(2^m-1)/2 < 4^m \le d-1$ pairs of vertices of $I^m(2)$, so there are fewer than $d-1$ types of bad events. 
Consider now a fixed bad event. The corresponding points, $p$ and $q$  
of $I^m(2)$ have different coordinates, 
assume without loss of generality that their first coordinates are different. 
We want to estimate the probability that the points of $C$ corresponding to $p$ and $q$ are of the same color. Fix all pairs of points in each of the simplices $S_d(2)$, except in the first one.
In the first one, fix only one of the points, say, $a$, so that the position of the point corresponding to $p$ is already fixed. 
We still have to take the other point, $b$, from the first simplex. 
As we go over all vertices of the first simplex, except $a$, the possible locations of the point of $C$ that corresponds to $q$ go over the vertices of a simplex $S_{d-1}(2)$. 
Since no copy of $I(2)$ is monochromatic, this copy of $S_{d-1}(2)$ is rainbow. So we get that
the bad event of the given type at most once.
This implies that the probability of each bad event is at most $1/(d-1)$. 
Now the union bound implies that a random copy of $I^m(2)$ is rainbow with a positive probability, which completes the proof. 
\end{proof}

Before the induction step, we need one more simple statement.

\begin{Proposition} \label{Prop3}
 For all $r, m \in \N$, there exists $n=n_3(r,m)$ such that $I^n(1) \xrightarrow{r} I^m(2)$.
\end{Proposition}
\begin{proof}
We prove the following technical statement by induction.

\vspace{1mm}
\noindent
\textit{For all $n_1, \dots , n_m$ such that $n_1>r, n_2> n_1^r, \dots, n_m > (n_1\cdot\dotsc\cdot n_{m-1})^r$, we have $S_{n_1}(2)\times\dots\times S_{n_m}(2) \xrightarrow{r} I^m(2)$.}
\vspace{1mm}

This implies the original statement since  $S_{n_1}(2)\times\dots\times S_{n_m}(2)\subset I^n(1)$ 
for $n=n_1+\dots+n_m$.
The $m=1$ case is trivial, since among $n_1>r$ vertices of $S_{n_1}(2)$ some two are of the same color by the pigeonhole principle. To prove the induction step, we consider $S_{n_1}(2)\times\dots\times S_{n_m}(2)$ as the union of $n_m$ `layers' of the form $S_{n_1}(2)\times\dots\times S_{n_{m-1}}(2)\times\{x\}$, $x \in S_{n_m}(2)$. Since $n_m > (n_1\cdot\dotsc\cdot n_{m-1})^r$, some two layers are colored identically by the pigeonhole principle. Now the union of a monochromatic copy of $I^{m-1}(2)$ from one of these layers, which must exist by the induction hypothesis, with the corresponding set from the other layer form the desired monochromatic copy of $I^{m}(2)$.
\end{proof}

Now we complete the proof of \Cref{Prop1} for $k>1$ by induction. We show that the desired statement holds for
\begin{equation*}
	n=n_1(k,m) \coloneqq n' + n_3\big(r', m'\big), \mbox{ where } n' = n_2(m) 2^{k-1}, \ r' = 2^{n'+k}, \ m'=n_1(k-1,m).
\end{equation*}
Fix an arbitrary coloring of $I^n(1)$.

For each $\x \in I^{n-n'}(1)$, we consider a `layer' $\{\x\}\!\times\! I^{n'}(1) \subset I^n(1)$. By partitioning the $n'$ basic vectors of $I^{n'}(1)$ into $n_2(m)$ classes of size $2^{k-1}$ and applying \Cref{Prop2}, we conclude that $I^{n'}(1) \MR (I(2^k), I^m(2^k))$. If at least for one $\x \in I^{n-n'}(1)$, the second alternative holds, that is, there exists a rainbow copy of $I^m(2^k)$ in the layer $\{\x\}\!\times\! I^{n'}(1)$, then we are done. So we can assume without loss of generality that for all $\x \in I^{n-n'}(1)$, there exists a monochromatic copy of $I(2^k)$ in the layer $\{\x\}\!\times\! I^{n'}(1)$. 

Consider the auxiliary $r'$-coloring of $I^{n-n'}(1)$, where each vertex $\x$ is colored according to which of the $r'$ copies of $I(2^k)$ in the layer $\{\x\}\!\times\! I^{n'}(1)$ is monochromatic.
If the layer contains several monochromatic copies of $I(2^k)$, we pick one of them arbitrarily. 
Now \Cref{Prop3} implies that there exists a copy of $I^{m'}(2)$ that is monochromatic under this auxiliary coloring. In terms of the original coloring, this gives us two identically colored copies of $I^{m'}(2)$ that form a `prism' $I^{m'}(2) \times I(2^k)$. 

Applying the induction hypothesis to the base of this prism, we find either a monochromatic copy of $I^{k-1}(2\cdot2^{k-1})$, or a rainbow copy of $I^m(2\cdot2^{k-1})$ there. In the latter case, we are done immediately. In the former case, the union of this monochromatic copy of $I^{k-1}(2^k)$ in the first level of the prism with the corresponding identically colored set from the second level form the desired monochromatic copy of $I^k(2^k)$. This completes the proof of \Cref{Prop1}.

\section{Maximum norm} \label{Sec3}

\subsection{Hales--Jewett theorem} \label{Sec_HJ}

In this subsection, we discuss the Hales--Jewett theorem, which is a central result in Ramsey theory and also the main tool of this section.
Informally speaking, it states that a sufficiently high dimensional grid, colored with a given number of colors,  contains a given dimensional, 
monochromatic subgrid of a certain type, described below.

For $k,n \in \N$, let $\tau$ be an element of $([k]\cup\{*\})^n$ with at least one $*$-coordinate, where $*$ is an abstract symbol that we call a \textit{variable}. We define a \textit{combinatorial line} corresponding to $\tau$ as a subset $\{\tau(i): i \in [k]\} \subset [k]^n$, where $\tau(i)$ is an element of $[k]^n$ that we obtain by replacing all the variable coordinates of $\tau$ by $i$. More generally, let $\tau$ be the element of $([k]\cup\{*_1,\dots,*_m\})^n$ with at least one variable coordinate of each of the $m$ types. We define an \textit{$m$-dimensional combinatorial subspace} corresponding to $\tau$ as a subset $\{\tau(i_1,\dots,i_m): i_1,\dots,i_m \in [k]\} \subset [k]^n$, where $\tau(i_1,\dots,i_m)$ is an element of $[k]^n$ that we obtain by replacing all the variable coordinates of $\tau$ of the $j$-th type by $i_j$ for each $j \in [m]$. In this notation, the multidimensional Hales--Jewett theorem is the following statement.

\begin{Theorem}[Multidimensional Hales--Jewett theorem] \label{theorem_MHJ}
	For all $k,r,m \in \N$, there is a sufficiently large $n \in \N$ such that for every $r$-coloring of $[k]^n$, there exists a monochromatic $m$-dimensional combinatorial subspace.
\end{Theorem}

Hales and Jewett~\cite{HJ63} proved this theorem for $m=1$ in 1963, while the multidimensional 
generalization is due to Graham and Rothschild~\cite{GR71}. 
Furstenberg and Katznelson\cite{FK91} proved a density versions of this theorem using 
ergodic theory, while in 2009, the Polymath Project~\cite{P12} developed a combinatorial proof. 

For the canonical version, if $k>2$, then it is not true that for every coloring of $[k]^n$, there exists either a monochromatic or a rainbow $m$-dimensional combinatorial subspace even if $n$ is arbitrary large. Indeed, let $\sim$ be an equivalence relation on $[k]$. Consider a coloring $c_{\sim}$ of $[k]^n$ such that for all $\x = (x_1,\dots,x_n)$ and $\y=(y_1,\dots,y_n)$ in $[k]^n$, we have
\begin{equation*}
	c_{\sim}(\x) = c_{\sim}(\y) \mbox{ \ if and only if \ } x_1\sim y_1, \dots , x_n \sim y_n.
\end{equation*}
Observe that the induced coloring on every $m$-dimensional combinatorial subspace of $[k]^n$ 
coincides with the coloring $c_{\sim}$ on $[k]^m$ under a natural isomorphism. 
Therefore, no $m$-dimensional combinatorial subspace is monochromatic or rainbow, unless $\sim$ is a \textit{trivial} 
equivalence relation, that is, 
the elements of $[k]$ are either pairwise equivalent or pairwise non-equivalent. 
However, Pr\"omel and Voigt~\cite[Theorem~C.7]{PV83} showed that one can always find a 
combinatorial subspace colored according to some $c_\sim$, see also~\cite{NPRV85}.

\begin{Theorem}[Canonical Hales--Jewett theorem] \label{theorem_CHJ}
	For all $k,m \in \N$, there is a sufficiently large $n \in \N$ such that for every coloring of $[k]^n$, 
 there exists an $m$-dimensional combinatorial subspace such that the induced coloring on it coincides 
 with $c_\sim$ for some equivalence relation $\sim$ on $[k]$.
\end{Theorem}

\subsection{Proof of Theorem~\ref{maxnorm1}}

In this section, we deduce \Cref{maxnorm1} from the Canonical Hales--Jewett theorem. 
Observe that our finite $M\subset \R^d$ is a subset of a Cartesian power $A^d$ for some finite $A \subset \R$. 
Indeed, let $A$ be the union of the projections of $M$ onto each of the $d$ coordinate axes. 
Hence, it is sufficient to prove \Cref{maxnorm1} only for $A^d$ playing the role of $M$.

Label the elements of $A$ by $a_1,\dots,a_s$ in ascending order. 
By adding between every two of its consecutive elements sufficiently many equally spaced intermediate points, 
we obtain a set $B=\{b_1,\dots,b_k\} \supset A$ such that
\begin{equation} \label{eq_small}
	\max_{1 \le i < k} (b_{i+1}-b_{i}) \le \min_{1\le j < s} (a_{j+1}-a_j),
\end{equation}
where the elements of $B$ are also labeled in ascending order, see \Cref{construction}.

Let $n$ be from the statement of \Cref{theorem_CHJ} applied to our $k$ and $m = d+\lceil d\log_2s\rceil$. 
In what follows, we prove that every coloring of $B^n$ contains either a monochromatic or a rainbow $\ell_\infty$-isometric copy of $A^d$, 
which implies that $\R_\infty^n \MR A^d$. 
Consider an arbitrary coloring of $B^n$. A natural bijection $\varphi$ between $[k]^n$ and $B^n$ defined by
\begin{equation*}
	\varphi(x_1,\dots,x_n) = (b_{x_1}, \dots, b_{x_n}) \mbox{ for all } (x_1,\dots,x_n) \in [k]^n
\end{equation*}
translates that coloring into a coloring of $[k]^n$. 
Now \Cref{theorem_CHJ} implies that the $m$-dimensional combinatorial subspace corresponding to some  
$\tau \in ([k]\cup\{*_1,\dots,*_m\})^n$ is colored according to $c_\sim$ for some equivalence relation $\sim$ on $[k]$.

If all the elements of $[k]$ are pairwise equivalent with respect to $\sim$, i.e., if the combinatorial space is monochromatic, 
then it is easy to complete the proof. Indeed, consider the subset of our combinatorial space where the last $m-d$ variables stay fixed, say, $*_{d+1}=\dots=*_m=1$, while the first $d$ variables range over the indices of the elements of $A$ inside $B$. 
It is not hard to see that the image of this subset under $\varphi$ 
is the desired monochromatic $\ell_\infty$-isometric copy of $A^d$ inside $B^n$. 
More precisely, let $1=i_1<\dots<i_s = k$ be a sequence such that $a_j=b_{i_j}$ for all $j \in [s]$, 
and put $I = \{i_1,\dots,i_s\} \subset [k]$, see \Cref{construction}. 
Let $\mathcal{I}_1$ be the set of all $s^d$ points of $[k]^m$ such that their first $d$ coordinates belong to $I$, 
while the remaining $m-d$ coordinates are equal to $1$. 
Note that $\tau(\mathcal{I}_1)\subset [k]^n$ is a subset of our combinatorial subspace, 
and thus it is monochromatic. Moreover, $\varphi(\tau((\mathcal{I}_1)) \subset B^n$ is an $\ell_\infty$-isometric copy of $A^d$, 
because the $\ell_\infty$-norm is independent of the arguments' multiplicities, i.e., the number of $*_r$-coordinates of $\tau$ for all $r \in [d]$, as desired. 

\begin{figure}[htp]
	\centering
	\includegraphics[width=10.5cm]{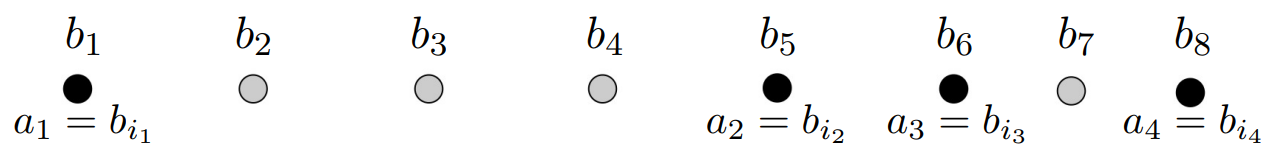}
	\caption{For $A=\{0,16,20,26\}$, we can take $B=\{0,4,8,12,16,20,23,26\}, \ I=\{1,5,6,8\}$.}
	\label{construction}
\end{figure}

Otherwise, there exists $1\le i' < k$ such that $i' \not\sim i'+1$. 
Let $\sigma$ be an arbitrary injection from $I^d$ to $\{i', i'+1\}^{m-d}$, where $I$ is from the previous paragraph. 
Such an injection exists because $|I^d| = s^d \le 2^{m-d} = |\{i', i'+1\}^{m-d}|$ by our choice of $m$. 
Consider $\mathcal{I}_2 = \{(\x, \sigma(\x)) : \x \in I^d\} \subset [k]^m$. Note that every two different 
elements of $\mathcal{I}_2$ have at least one non-equivalent coordinate, and thus the set $\tau(\mathcal{I}_2) \subset [k]^n$ is rainbow. 
Moreover, it is not hard to check that $\varphi(\tau(\mathcal{I}_2)) \subset B^n$ is an $\ell_\infty$-isometric copy of $A^d$. 
Indeed, in the previous paragraph, we have already seen this when the last $m-d$ variable coordinates were fixed. 
Now they are not fixed, but vary between $i'$ and $i'+1$. However, \eqref{eq_small} 
implies that $b_{i'+1}-b_{i'}$ is not larger than the smallest $\ell_\infty$-distance between different elements of $A^d$. 
Therefore, these additional variable coordinates do not change any of the 
$\ell_\infty$-distances. 
This completes the proof \Cref{maxnorm1}.

\section{Concluding remarks} \label{Sec4}

\noindent
\textbf{Canonical types.} 
For any finite set $A$, let $\R^n\MRR A$ denote the claim that every 
$r$-coloring of $\R^n$ contains a monochromatic or a rainbow copy of 
$A$. Clearly, if $\R^n \MRR A$ then also $\R^{n+1}\MRR A$ and $\R^n \smash{ \xrightarrow[\smash{\raisebox{0.7ex}{\tiny $\mathbf{MR}$}}]{r-1}} A$. Therefore, we can distinguish the following three types of sets, see \Cref{3cases}.

\vspace{-2mm}

\begin{enumerate} \setlength{\itemsep}{0pt}
	\item[1.] There is $r$ such that $\R^n\nMRR A$ for every $n$.
	\item[2.] There is $n$ such that $\R^n\MRR A$ for every $r$.
	\item[3.] None of the above, i.e., for every $r$ there is $n$ with  $\R^n\MRR A$, 
 and for every $n$ there is $r$ with $\R^n\nMRR A$.
\end{enumerate} 

\begin{figure}[htp]
\centering
\includegraphics[width=16cm]{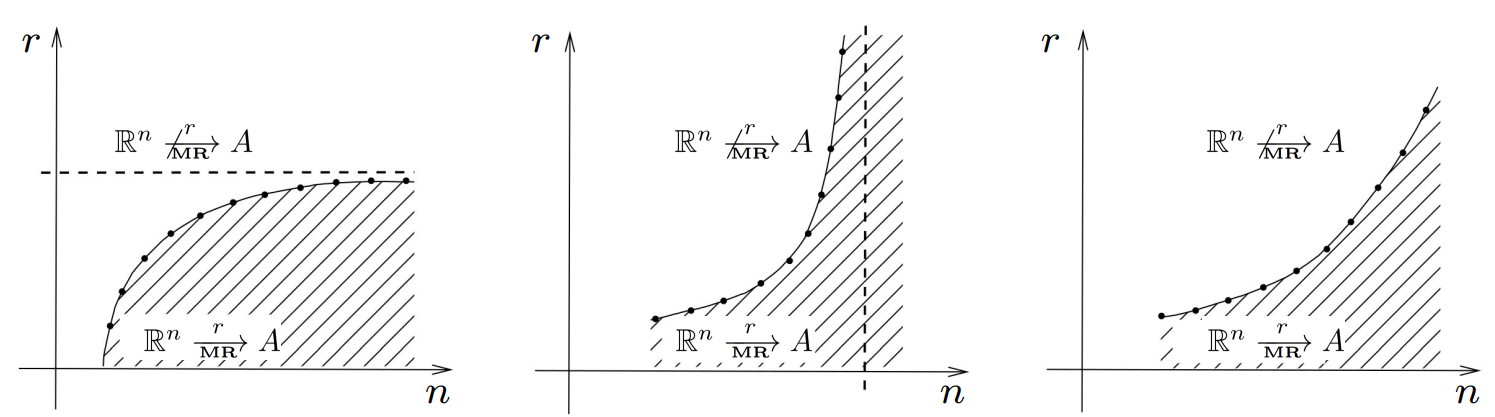}
\caption{All three possible types of finite sets.}
\label{3cases}
\end{figure}

It is clear that Ramsey sets cannot be of type 1, and it is natural to ask whether they are all of type 2. In fact, we cannot fully answer this question even if the Ramsey set $A$ is a rectangle or an obtuse triangle. Our \Cref{cube} implies that rectangles with sidelengths  $a$ and $b$ are of type 2, provided that $(b/a)^2$ is rational, but we could not determine a type of the remaining rectangles. As we mentioned in the Introduction, Cheng and  Xu~\cite{CX23} proved that the obtuse triangles that are not `too flat' are of type 2. The type of 'flat' obtuse triangles remains uncertain. We do not have an example of a Ramsey set of type 3.

For non-Ramsey sets we cannot rule out that all three types are possible. 
As it was observed in~\cite{CX23}, a classical spherical $3$-coloring 
of $\R^n$ without monochromatic copies of a $4$-term unit arithmetic progression $l_4$ from~\cite{EGMRSS1} 
cannot contain a rainbow copy either, so $l_4$ is of type 1. We are not aware of any non-Ramsey set of other types. A $3$-term unit arithmetic progression $l_3$ might be natural candidate for type 2 
(with $n$ being $2$ or $3$), since it was recently shown in \cite{CMU24} that $\R^2 \xrightarrow{2} l_3$. A similar question was raised in \cite[Section~4]{MP24}.

\vspace{2mm}

\noindent
\textbf{Tightness of \Cref{theorem_triangle}.}
A classical $2$-coloring of the plane with color-alternating strips of heights $\sqrt{3}/2$ does not contain monochromatic copies of a unit equilateral triangle $\triangle$. Hence, it is \textit{not} true that $\R^2 \MR \triangle$. Are there other triangles $T$ for which the dimension $3$ in \Cref{theorem_triangle} is also tight? It was conjectured in \cite[Conjecture~3]{EGMRSS3} that $\R^2 \xrightarrow{2} T$ for every non-equilateral triangle $T$. If true, this would be an evidence supporting that $\R^2 \MR T$ might hold for all non-equilateral triangles $T$, and thus our result is not tight. However, during the last 50 yeas, that conjecture was verified only for a few special families of triangles, see~\cite{NS23} and the references therein.

\vspace{2mm}

\noindent
\textbf{Quantitative bounds in \Cref{cube}.} What is the minimum $n=n(m)$ such that $\R^n \MR I^m$? A careful analysis of our proof leads to the upper bound $n(m) = \exp_2^{2m-1}\big(O(m)\big)$, where $\exp_2^{1}(x)=2^x$ and $\exp_2^{k+1}(x) = 2^{\exp_2^{k}(x)}$ for all $k>1$. However, our argument is very wasteful in a sense that all the auxiliary propositions are stated for the hypercubes, but in fact we are working only with their small subsets: Cartesian product of simplices. Taking this into account, we can improve the upper bound to $n(m) = \exp_2^{4}\big(O(m)\big)$, which is still a quadruple exponent. From the other direction, we know that $n(m) = \Omega(m \log m)$, since for smaller $n$, there exists a coloring of $\R^n$ with no monochromatic isometric copies of $I^m$ that uses less than $2^m$ colors, see~\cite{Pros18}.

The case $m=2$ is of a particular interest here. We can show that $K_{876}\square K_{123} \square K_{4} \square K_{2} \MR K_{2} \square K_{2}$, where $\square$ stands for the Cartesian product of these cliques. A straightforward translation of this purely graph theoretical statement into geometric language implies that $n(2) \le 875+122+3+1=1001$. We do not know what is the best upper bound on $n(2)$ that can be derived from~\cite{CX23}, but it might be of the same order. From the other direction, we do not even know if $n(2)>3$.

\vspace{2mm}

\noindent
\textbf{Max-norm Ramsey theory.} Though our \Cref{maxnorm1} is rather general, the quantitative bounds on the minimum dimension $n=n(M)$ such that  $\R_\infty^n \MR M$ it generates are probably extremely far from being tight. For instance, if $M$ is a rectangle with sidelengths $a<b$, then the number of auxiliary points on \Cref{construction}, as well as the resulting upper bound, will depend on the fraction $b/a$. We found an ad hoc argument for this case in the spirit of~\cite{CX23} which leads to the upper bound $O(b/a)$, i.e., this dependence is at most linear. Does there exist a universal constant $n$ such that $\R_\infty^n \MR R$ for every rectangle $R$? 

\vspace{2mm}

\noindent
\textbf{Manhattan Ramsey theory.} Let $\R_1^n$ be the $n$-dimensional space equipped with the Manhattan $\ell_1$-norm. It is not hard to show that for every $r$ and a finite $M \subset \R^d$, there exists $n$ such that $\R_1^n \xrightarrow{r} M$, see~\cite{KSZ23}. Our \Cref{Prop1} has the following canonical corollary: for every finite $M \subset \Q^d$, there exists $n$ such that $\R_1^n \MR M$. Indeed, it is sufficient to note that a path of $m$ pairwise orthogonal edges in the hypercube is an $\ell_1$-isometric copy of an arithmetic progression. Besides that, every `rational' $M$ can be embedded into a Cartesian power of this progression after a proper scaling. We wonder, if the same conclusion holds for `non-rational' sets $M$ as well, e.g.\ for all rectangles.

\vspace{3mm}

\noindent
{\bf \large Acknowledgments.}
We are very grateful to Endre Szab\'o and Bal\'azs Csik\'os for their help in the presentation of Lemma 1. Arsenii Sagdeev thanks Alexander Golovanov for  the helpful discussion of Lemma 1.

The authors were supported by ERC Advanced Grant `GeoScape' No.\ 882971. Panna Geh\'er was also supported by the Lend\"ulet Programme of the Hungarian Academy of Sciences -- grant number LP2021-1/2021 and by the Thematic Excellence Program TKP2021-NKTA-62 of the National Research, Development and Innovation Office. G\'eza T\'oth was also supported by the National Research, Development and Innovation Office, NKFIH, K-131529.

\end{document}